\documentclass[12pt,reqno]{amsart}   	
\usepackage{bm} 
\usepackage{amsfonts}
\usepackage{mathrsfs}
\usepackage{yfonts}
\usepackage{url}

\usepackage[osf]{mathpazo}  

\usepackage{mathtools}
\usepackage{amssymb}  		
\usepackage{graphicx}
\usepackage{verbatim}
\usepackage{enumerate}	
\usepackage[english]{babel}
\usepackage{amsmath}
\usepackage{bussproofs}	
\usepackage{quoting}
\usepackage{array,booktabs}

\usepackage[shortlabels]{enumitem}
\usepackage{units}
\usepackage{xspace}\xspaceaddexceptions{\,}
\usepackage{multicol}
\usepackage{booktabs,caption}
   \usepackage{tikz-cd}
\quotingsetup{font=small}	
\usepackage{amsthm}								
\usepackage
[authormarkup=none]
{changes}

\definechangesauthor[color=red]{Stef}

\newcommand*\circled[1]{\tikz[baseline=(char.base)]{
            \node[shape=circle,draw,inner sep=4pt] (char) {#1};}}

\usepackage[pdftex,bookmarks,bookmarksnumbered,linktocpage,   
         colorlinks,linkcolor=blue,citecolor=blue]{hyperref}

\theoremstyle{definition}
\newtheorem{definition}{Definition}
\newtheorem{example}[definition]{Example}

\newtheorem{remark}[definition]{Remark}

\theoremstyle{plain}
\newtheorem{lemma}[definition]{Lemma}

\newtheorem{theorem}[definition]{Theorem}
\newtheorem{corollary}[definition]{Corollary}

\newcommand\A{{\mathbf A}}
\newcommand\B{{\mathbf B}}
\newcommand\C{{\mathbf C}}

\newcommand\Fm{\mathbf{Fm}}

\newcommand\WK{{\mathbf{WK}}}

\newcommand\CL{\ensuremath{\mathrm{CL}}\xspace}

\newcommand\PWK{\ensuremath{\mathrm{PWK}}\xspace}

\bmdefine{\boldstar}{\mathchoice{\textstyle*}{\textstyle*}{\textstyle*}{\scriptstyle*}}

\newcommand\Alg[1]{\if#1*\operatorname{\mathsf{Alg}*}\else\operatorname{\mathsf{Alg}}#1\fi}

\newcommand\Mod[1]{\if#1*\operatorname{\mathsf{Mod}*}\else\operatorname{\mathsf{Mod}}#1\fi}

\bmdefine{\Leibniz}{\Omega}        

\bmdefine{\frege}{\Lambda}         

\makeatletter
\newcommand{\tarskidsp}{\mathord%
   {\m@th\raisebox{0pt}[0pt][0pt]{$\stackrel%
   {\raisebox{-2.7pt}[0ex][0pt]{$\displaystyle \,\?\thicksim$}}%
   {\displaystyle\Leibniz}$}}}
\newcommand{\tarskitxt}{\mathord%
   {\m@th\raisebox{0pt}[0pt][0pt]{$\stackrel%
   {\raisebox{-2.7pt}[0ex][0pt]{$\,\?\thicksim$}}{\displaystyle\Leibniz}$}}}
\newcommand{\tarskiscr}{\mathord%
   {{\m@th\raisebox{0pt}[0pt][0pt]{$\stackrel%
   {\raisebox{-2.4pt}[0ex][0pt]{$\scriptstyle \,\?\thicksim$}}%
   {\scriptstyle\Leibniz}$}}}}
\newcommand{\tarskiscrscr}{\mathord%
   {{\m@th\raisebox{0pt}[0pt][0pt]{$\stackrel%
   {\raisebox{-2pt}[0ex][0pt]{$\scriptscriptstyle \,\?\thicksim$}}%
   {\scriptscriptstyle\Leibniz}$}}}}
\newcommand{\Tarski}{\@ifnextchar ^ %
   {\mathchoice{\tarskidsp\kern-.07em}{\tarskitxt\kern-.07em}%
   {\tarskiscr\kern-.07em}{\tarskiscrscr\kern-.07em}}%
   {\mathchoice{\tarskidsp}{\tarskitxt}{\tarskiscr}{\tarskiscrscr}}}
\makeatother
\DeclareMathAlphabet{\mathbfsf}{\encodingdefault}{\sfdefault}{bx}n
\makeatletter
\providecommand*{\Dashv}{\mathrel{\mathpalette\@Dashv\vDash}}
\newcommand*{\@Dashv}[2]{\reflectbox{$\m@th#1#2$}}
\makeatother

\renewcommand\geq{\geqslant}

\useshorthands{"}
\defineshorthand{"-}{{}\nobreakdash-\hspace{0pt}}	

\newcommand\PL{{\mathcal{P}}_{\text{\l}}}

\EnableBpAbbreviations

\newcommand{\bit}{\begin{itemize}}    
\newcommand{\eit}{\end{itemize}}
\newcommand{\ben}{\begin{enumerate}}
\newcommand{\een}{\end{enumerate}}

\newcommand{\benroman}{\ben[\normalfont (i)]}  
\let\eroman\een
\newcommand{\benbullet}{\ben[\textbullet]}     
\let\ebullet\een
\newcommand{\bde}{\begin{description}}
\newcommand{\ede}{\end{description}}

\newcommand{\Var}{\mathnormal{V\mkern-.8\thinmuskip ar}} 
\newcommand{\?}{\ensuremath{\mkern0.4\thinmuskip}}   

\begin{document}

\title{Logics of variable inclusion and the lattice of consequence relations}
\date{\today}

\author{Michele Pra Baldi}
\address{Michele Pra Baldi, University of Cagliari, Italy}
\email{m.prabaldi@gmail.com}
\begin{abstract}
In this paper, firstly, we determine the number of sublogics of variable inclusion of an arbitrary finitary logic $\vdash$ with partition function. Then, we investigate their position into the lattice of consequence relations over  the language of $\vdash$.

\end{abstract}
\maketitle

\section{introduction}

The family of logics of variable inclusion splits into two subfamilies, namely \emph{logics of left variable inclusion} and \emph{logics of right variable inclusion}. More precisely, given a logic $\vdash$, the two sublogics that can be defined by means of a different \emph{variable inclusion} principle are

\[
\Gamma \vdash^{l} \varphi \Longleftrightarrow \text{ there is }\Delta \subseteq \Gamma \text{ s.t. }\Var(\Delta)\subseteq\Var(\varphi) \text{ and } \Delta\vdash\varphi,
\]
and
\[
\Gamma\vdash^{r}\varphi\iff \left\{ \begin{array}{ll}
\Gamma\vdash\varphi  \ \text{and} \ \Var(\varphi)\subseteq\Var(\Gamma)& \text{or}\\
\Sigma\subseteq\Gamma, & \\
  \end{array} \right.  
\]

where $\Sigma$ is an antitheorem of $\vdash$ (see Definition \ref{def inconsistency terms}).

Here, the logic $\vdash^l$ denotes the left variable inclusion companion of $\vdash$, while $\vdash^r$ is its right variable inclusion counterpart. 
The best known examples of variable inclusion logics arise when $\vdash$ is considered to be classical logic. In this case, $\vdash^l$ is known as \textit{paraconsistent weak Kleene logic} (\PWK for short) \cite{Kleene,Hallden} and $\vdash^r$ as \textit{Bochvar logic} ($\mathsf{B}_{3}$)\cite{Bochvar,Kleene,Hallden}. These two logics are semantically defined on the base of the so-called weak Kleene tables 

\medskip

\begin{center}\renewcommand{\arraystretch}{1.25}
\begin{tabular}{>{$}c<{$}|>{$}c<{$}>{$}c<{$}>{$}c<{$}}
   \land & 0 & n & 1 \\[.2ex]
 \hline
       0 & 0 & n & 0 \\
       n & n & n & n \\          
       1 & 0 & n & 1
\end{tabular}
\qquad
\begin{tabular}{>{$}c<{$}|>{$}c<{$}>{$}c<{$}>{$}c<{$}}
   \lor & 0 & n & 1 \\[.2ex]
 \hline
     0 & 0 & n & 1 \\
     n & n & n & n \\          
     1 & 1 & n & 1
\end{tabular}
\qquad
\begin{tabular}{>{$}c<{$}|>{$}c<{$}}
  \lnot &  \\[.2ex]
\hline
  1 & 0 \\
  n & n \\
  0 & 1 \\
\end{tabular}
\end{center}
\vspace{10pt}

as follows:
\benbullet

\item $\langle\WK,\{1\}\rangle=\mathsf{B_{3}}$
\item $\langle\WK\{1,n\}\rangle=\PWK$, where $\WK$ is the three elements algebra induced by the above tables.

\ebullet

Logics of variable inclusion have recently been influential in several research areas, including the philosophy of language \cite{beall2016off}, theories of truth \cite{Szmuctruth} and, of course, logic \cite{cc,Szmuc,BonzioMorascoPrabaldi,BonzioPraBaldi,Bonzio16}.
On the logical side, the fact that $\PWK$ actually corresponds to the left variable inclusion companion of classical logic is shown in \cite{CiuniCarrara}, while \cite{Bonzio16} contains an algebraic study of \PWK with the tools of modern abstract algebraic logic (AAL).  
The work in \cite{BonzioMorascoPrabaldi}, which  also adopts the AAL framework, identifies a general method to
turn a complete matrix semantics for an arbitrary logic $\vdash$ into a complete matrix semantics for its left
variable inclusion companion.
 A similar task is accomplished in \cite{BonzioPraBaldi} for finitary right variable inclusion logics. 

Of course, nothing prevents from iterating the definitions of left and right variable inclusion logics. For instance, one can define the logic $\vdash^{lr}$, that is the right variable inclusion companion of the left variable inclusion companion of $\vdash$. The only known example of this kind is the logic $\mathsf{K^{w}_{4n}}$, investigated in the very recent papers \cite{Tomova,Petrukhin2}.
In general, by looking at the above definitions, it is immediate to verify that each logic of variable inclusion of $\vdash$ is a sublogic of $\vdash$.

 The general theory of closure operators states that, given a set $A$, the set of all the structural closure operators on $A$ can be equipped with a (complete) lattice structure. One of the outcomes of the pioneering work of \cite{BP89} and of the more recent developments in abstract algebraic logic contained in \cite{Font16, blok2006equivalence, font2015m} states that  there is a bijective correspondence between logics in the language $\mathcal{L}$ and structural closure operators over the set of formulas $Fm_\mathcal{L}$ equipped with a monoid action (whose elements represents substitutions). This perspective highlights that the investigation of the lattice of logics over a fixed language $\mathcal{L}$ is
worth pursuing.

In \cite{Szmuc}, a first attempt to determine how $\mathsf{B}_3$ and \PWK relates with other sublogics of \CL is offered. However, a general and systematic method that determines how the logics of variable inclusion of $\vdash$ fit
into the lattice of logics over $\mathcal{L}$ (with $\vdash$ being a finitary logic over a fixed language $\mathcal{L}$) is still missing.

The main aim of this paper is to fill this gap, by solving the above mentioned problem in full generality.  It will turn out that the number of sublogics of variable inclusion of a logic $\vdash$ is no greater that 8 if $\vdash$ possesses an antitheorem, and no greater that 5 otherwise.
In the final section, we consider the example of classical logic, and we describe in a transparent way the relations among  its sublogics of variable inclusion. Remarkably, it turns out that only four of these eight logics have been considered in the literature until now.

\section{Preliminaries}\label{sec: preliminari}

For standard background on closure operators and abstract algebraic logic we refer the reader respectively to  \cite{BuSa00,bergmann}and \cite{BP89,Cz01,Font16}. Unless stated otherwise, we work within a fixed but arbitrary algebraic language. We denote algebras by $\A, \B, \C\dots$ respectively with universes $A, B, C \dots$ 

\subsection{Abstract algebraic logic}

Let $\Fm$ be the algebra of formulas built up over a countably infinite set $\Var$ of variables. Given a formula $\varphi\in Fm$, we denote by $\Var(\varphi)$ the set of variables really occurring in $\varphi$. Similarly, given $\Gamma\subseteq Fm$, we set
\[
\Var(\Gamma)=\bigcup \{\Var(\gamma)\colon \gamma\in\Gamma\}.
\]
A \emph{logic} is a substitution invariant consequence relation $\vdash \?\? \subseteq \mathcal{P}(Fm) \times Fm$ in the sense that for every substitution $\sigma \colon \Fm \to \Fm$,
\[
\text{if }\Gamma \vdash \varphi \text{, then }\sigma [\Gamma] \vdash \sigma (\varphi).
\]
 A logic $\vdash$ is \emph{finitary} when the following holds for all $\Gamma\cup\varphi\subseteq Fm$:
\begin{align*}
\Gamma\vdash\varphi \Longleftrightarrow \exists \Delta \subseteq\Gamma \text{ s.t. } \Delta \text{ is finite and } \Delta\vdash\varphi.
\end{align*}

 A \emph{matrix} is a pair $\langle \A, F\rangle$ where $\A$ is an algebra and $F \subseteq A$. In this case, $\A$ is called the \textit{algebraic reduct} of the matrix $\langle \A, F \rangle$.
 Every class of matrices $\mathsf{M}$ induces a logic as follows:
\begin{align*}
\Gamma \vdash_{\mathsf{M}} \varphi \Longleftrightarrow& \text{ for every }\langle \A, F \rangle \in \mathsf{M} \text{ and hom. }h \colon \Fm \to \A,\\
& \text{ if }h[\Gamma] \subseteq F\text{, then }h(\varphi) \in F.
\end{align*}
A logic $\vdash$ is \emph{complete} w.r.t.\ a class of matrices $\mathsf{M}$ when it coincides with $\vdash_{\mathsf{M}}$. 

 A matrix $\langle \A, F\rangle$ is a \emph{model} of a logic $\vdash$ when
\begin{align*}
\text{if }\Gamma \vdash \varphi, &\text{ then for every hom. }h \colon \Fm \to \A,\\  
&\text{ if }h[\Gamma] \subseteq F\text{, then }h(\varphi) \in F.
\end{align*}
A set $F \subseteq A$ is a (deductive) \textit{filter} of $\vdash$ on $\A$, or simply a $\vdash$-\textit{filter}, when the matrix $\langle \A, F \rangle$ is a model of $\vdash$. 
We denote the class of matrix models of $\vdash$ as $\Mod(\vdash)$.

The following definition originates in \cite{LaPr17}, but see also \cite{CampercholiRaftery,JGR13}
\begin{definition}\label{def inconsistency terms}
A set of formulas $\Sigma$  in an antitheorem of a logic $\vdash$ if $\sigma[\Sigma] \vdash \varphi$ for every substitution $\sigma$ and formula $\varphi$.
\end{definition}

\begin{example}\label{Exa:inconsistency-terms}
For any formula $\varphi$, the set $\{ \lnot(\varphi \to \varphi) \}$ is an antitheorem of all superintuitionistic logics. Similarly, the sets $\{\varphi,\neg\varphi\},\{\varphi\land\neg\varphi\}$ are antitheorems of all the expansions of Classical logic and Strong Kleene logic. 
\qed
\end{example}

\begin{remark}\label{rem: inconsistency term in var x only}
Observe that if $\vdash$ has an antitheorem, then $\vdash$ has an antitheorem only in variable $x$. If, moreover, $\vdash$ is finitary, then it has a \emph{finite} antitheorem only in variable $x$. 
\qed
\end{remark}
Given two logics $\vdash,\vdash^\prime$ in the same language, we say that  $\vdash^\prime$ is a $sublogic$ of $\vdash$ (in symbols $\vdash^\prime\leq\vdash$) if for every $\Gamma\cup\{\varphi\}\subseteq Fm$,

\[
\Gamma\vdash^\prime\varphi \  \text{ entails} \  \Gamma\vdash\varphi.
\]

Let $\mathcal{L}$ be an algebraic language. The The set $\mathcal{L}^\vdash$ of all logics in the language $\mathcal{L}$ forms a complete lattice (see \cite{W88} for details), where, given  $\vdash_i$, $i\in I$ logics over $\mathcal{L}$, the operations are defined as follows
\begin{align*}
\bigwedge_{i\in I}\vdash_i&\coloneqq\bigcap_{i\in I}\vdash_i\\
\bigvee_{i\in I}\vdash_i&\coloneqq\bigcap\{\vdash:\ \vdash\geq\vdash_i \text{for every} \ i\in I\}.
\end{align*}


An immediate consequence is that, given a logic $\vdash\in\mathcal{L}^\vdash$, the set of sublogics of $\vdash$ is a sublattice of $\mathcal{L}^\vdash$. Given a logic $\vdash$, we denote the set of its logics of variable inclusion by $\mathcal{SV}(\vdash)$.

\subsection{P\l onka sums}

The main mathematical tool that allows for a systematic study of logics of variable inclusion is an algebraic construction coming from universal algebra, and more specifically from the study of regular varieties, i.e. varieties of algebras satisfying only equations $\sigma\thickapprox\delta$ in which $\Var(\sigma)=\Var(\delta)$. Such construction, known as \emph{P\l onka sums}, originates in the late 1960's from a series of papers published by the polish mathematician J.P\l onka, who first provided a general representation theorem for regular varieties.

For standard information on P\l onka sums we refer the reader to \cite{Plo67a, Plo67,Romanowska92, romanowska2002modes}. A \textit{semilattice} is an algebra $\A = \langle A, \lor\rangle$, where $\lor$ is a binary commutative, associative and idempotent operation. Given a semilattice $\A$ and $a, b \in A$, we set
\[
a \leq b \Longleftrightarrow a \lor b = b.
\]
It is easy to see that $\leq$ is a partial order on $A$.
\begin{definition}\label{Def:Directed system of algebras}
A \textit{direct system of algebras} consists in 
\benroman
\item a semilattice $I = \langle I, \lor\rangle$;
\item a family of algebras $\{ \A_{i} : i \in I \}$ with disjoint universes;
\item a homomorphism $f_{ij} \colon \A_{i} \to \A_{j}$, for every $i, j \in I$ such that $i \leq j$;
\eroman
moreover, $f_{ii}$ is the identity map for every $i \in I$, and if $i \leq j \leq k$, then $f_{ik} = f_{jk} \circ f_{ij}$.
\end{definition}

Let $X$ be a direct system of algebras as above. The \textit{P\l onka sum} of $X$, in symbols $\PL(X)$ or $\PL(\A_{i})_{i \in I}$, is the algebra defined as follows. The universe of $\PL(\A_{i})_{i \in I}$ is the union $\bigcup_{i \in I}A_{i}$. Moreover, for every $n$-ary basic operation $f$ and $a_{1}, \dots, a_{n} \in \bigcup_{i \in I}A_{i}$, we set
\[
f^{\PL(\A_{i})_{i \in I}}(a_{1}, \dots, a_{n}) \coloneqq f^{\A_{j}}(f_{i_{1} j}(a_{1}), \dots, f_{i_{n} j}(a_{n}))
\]
where $a_{1} \in A_{i_{1}}, \dots, a_{1} \in A_{i_{n}}$ and $j = i_{1} \lor \dots \lor i_{n}$.\ 

Observe that if in the above display we replace $f$ by any complex formula $\varphi$ in $n$-variables, we still have that
\[
\varphi^{\PL(\A_{i})_{i \in I}}(a_{1}, \dots, a_{n}) = \varphi^{\A_{j}}(f_{i_{1} j}(a_{1}), \dots, f_{i_{n} j}(a_{n})).
\]

The theory of P\l onka sums is strictly related with a special kind of operation:

\begin{definition}\label{def: partition function}
Let $\A$ be an algebra of type $\nu$. A function $\cdot\colon A^2\to A$ is a \emph{partition function} in $\A$ if the following conditions are satisfied for all $a,b,c\in A$, $ a_1 , ..., a_n\in A^{n} $ and for any operation $g\in\nu$ of arity $n\geqslant 1$.
\begin{enumerate}[label=\textbf{P\arabic*}., leftmargin=*]
\item $a\cdot a = a$
\item $a\cdot (b\cdot c) = (a\cdot b) \cdot c $
\item $a\cdot (b\cdot c) = a\cdot (c\cdot b)$
\item $g(a_1,\dots,a_n)\cdot b = g(a_1\cdot b,\dots, a_n\cdot b)$
\item $b\cdot g(a_1,\dots,a_n) = b\cdot a_{1}\cdot_{\dots}\cdot a_n $
\end{enumerate}
\end{definition}

The next result makes explicit the relation between P\l onka sums and partition functions:

\begin{theorem}\cite[Thm.~II]{Plo67}\label{th: Teorema di Plonka}
Let $\A$ be an algebra of type $\nu$ with a partition function $\cdot$. The following conditions hold:
\begin{enumerate}
\item $A$ can be partitioned into $\{ A_{i} : i \in I \}$ where any two elements $a, b \in A$ belong to the same component $A_{i}$ exactly when
\[
a= a\cdot b \text{ and }b = b\cdot a.
\]
Moreover, every $A_{i}$ is the universe of a subalgebra $\A_{i}$ of $\A$.
\item The relation $\leq$ on $I$ given by the rule
\[
i \leq j \Longleftrightarrow \text{ there exist }a \in A_{i}, b \in A_{j} \text{ s.t. } b\cdot a =b
\]
is a partial order and $\langle I, \leq \rangle$ is a semilattice. 
\item For all $i,j\in I$ such that $i\leq j$ and $b \in A_{j}$, the map $f_{ij} \colon A_{i}\to A_{j}$, defined by the rule $f_{ij}(x)= x\cdot b$ is a homomorphism. The definition of $f_{ij}$ is independent from the choice of $b$, since $a\cdot b = a\cdot c$, for all $a\in A_i$ and $c\in A_j$.
\item $Y = \langle \langle I, \leq \rangle, \{ \A_{i} \}_{i \in I}, \{ f_{ij} \! : \! i \leq j \}\rangle$ is a direct system of algebras such that $\PL(Y)=\A$.
\end{enumerate}
\end{theorem}

It is worth remarking that the construction of Plonka sums preserves the validity of the so-called \textit{regular identities}, i.e. identities of the form $ \varphi \thickapprox \psi $ such that $\Var(\varphi) = \Var(\psi)$.

\section{Matrix models for logics of variable inclusion}

In this section we review how to generalize the machinery of P\l onka sums up to logical matrices, in order to provide a complete matrix semantics for an arbitrary, finitary logic of variable inclusion.

\subsection{Left variable inclusion logics}

The definition of \emph{direct system} of algebras can be extended, as follows, to logical matrices:

\begin{definition}\label{Def:Directed-System-Matrices}(Essentially \cite[Definition 8]{BonzioMorascoPrabaldi})

A \textit{l-direct system} of matrices consists in 
\benroman
\item a semilattice $I = \langle I, \lor\rangle$;
\item a family of matrices $\{ \langle\A_{i},F_{i}\rangle \}_{i \in I}$ with disjoint universes;
\item a homomorphism $f_{ij} \colon \A_{i} \to \A_{j}$ such that $f_{ij}[F_{i}] \subseteq F_{j}$, for every $i, j \in I$ such that $i \leq j$
\eroman
 such that $f_{ii}$ is the identity map for every $i \in I$, and if $i \leq j \leq k$, then $f_{ik} = f_{jk} \circ f_{ij}$.
\end{definition}

Given a $l$-direct system of matrices $X$ as above, we set
\[
\PL(X) \coloneqq \langle \PL(\A_{i})_{i \in I}, \bigcup_{i \in I}F_{i}\rangle.
\]

The matrix $\PL(X)$ is the \textit{P\l onka sum} of the $l$-direct system of matrices $X$. Given a class $\mathsf{M}$ of matrices, we denote by $\PL^l(\mathsf{M})$ the class of all P\l onka sums of $l$-direct systems of matrices in $\mathsf{M}$.

The following Theorem establishes a completeness results for left variable inclusion logics.
\begin{theorem}(\cite[Theorem 14]{BonzioMorascoPrabaldi})\label{thm:completeness for left}
Let $\vdash$ be a logic and $\mathsf{M}$ be a class of matrices containing $\langle \boldsymbol{n}, \{ n \} \rangle$. If $\vdash$ is complete w.r.t.\ $\mathsf{M}$, then $\vdash^{l}$ is complete w.r.t.\ $\PL^l(\mathsf{M})$.
\end{theorem}

\begin{example}\label{exa: pwk as left of cl}

As paradigmatic application of the above theorem, consider  the case in which $\vdash=\vdash_\CL$. Consider the class of matrices $\{\langle\B_2,1\rangle,\langle\mathbf{n},n\rangle\}$, where $\B_2$ is the two-element Boolean algebra, and $\mathbf{n}$ is the trivial algebra. Theorem \ref{thm:completeness for left} states that the following matrix is complete for $\vdash_\PWK$
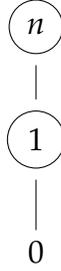
\begin{center}

\begin{tikzpicture}

  \node (n) at (-1.5,5) {\circled{$n$}};
    \node (1) at (-1.5,3.5) {\circled{$1$}};
    \node (0) at (-1.5,2) {$0$};
         \draw (0) -- (1); 
 \draw (1) -- (n);

\end{tikzpicture}
\end{center}

\end{example}
The following definition plays a central role in the algebraic study of logics of left variable inclusion.

\begin{definition}\label{def: logic with l p-func.}
A logic $\vdash$ has a \emph{l-partition function} if there is a formula $x\cdot y$, in which the variables $x$ and $y$ really occur, such that $x \vdash x\cdot y$ and the equations $\mathbf{P1.}, \dots,  \mathbf{P5.}$ in Definition \ref{def: partition function} hold in $\Alg(\vdash)$ for every $n$-ary connective $f$. In this case, $x \cdot y$ is a \emph{l-partition function} for $\vdash$.
\end{definition}

\begin{remark}\label{exa: left p function}
Observe that logics with a $l$-partition function abounds in the literature (see \cite{BonzioMorascoPrabaldi}). For instance, the term $x\land(x\lor y)$ is a $l$-partition function for the above mentioned logic $\PWK$.
\end{remark}
\subsection{Right variable inclusion logics}

\emph{Right variable inclusion logics}, also called \emph{containment logics} \cite{Parry}, are defined as follows:

\begin{definition}
Let $\vdash$ be a logic, $\vdash^{r}$ is the logic defined as 

\[
\Gamma\vdash^{r}\varphi\iff \left\{ \begin{array}{ll}
\Gamma\vdash\varphi  \ \text{and} \ \Var(\varphi)\subseteq\Var(\Gamma)& \text{or}\\
\Sigma(x)\subseteq\Gamma & \\
  \end{array} \right.  
\]

where $\Sigma(x)$ is an antitheorem of $\vdash$.
\end{definition}

Another possible way of extending the notion of direct system of algebras (see Definition \ref{Def:Directed system of algebras}) to logical matrices is the following:
\begin{definition}\label{def r direct sys}(Essentially \cite[Definition 13]{BonzioPraBaldi}\label{Def:Directed-System-Matrices}

A \textit{r-direct system} of matrices consists in 
\benroman
\item A semilattice $I = \langle I, \lor\rangle$.
\item A family of matrices $\{ \langle\A_{i},F_{i}\rangle : i \in I \}$ such that \\ $I^{+}\coloneqq\{i\in I:\langle\A_{i},F_{i}\rangle:F_{i}\neq\emptyset\} $ is a  sub-semilattice of $I$.
\item a homomorphism $f_{ij} \colon \A_{i} \to \A_{j}$, for every $i, j \in I$ such that $i \leq j$, satisfying also that: 
\begin{itemize}
\item $f_{ii}$ is the identity map for every $i \in I$;
\item if $i \leq j \leq k$, then $f_{ik} = f_{jk} \circ f_{ij}$;
\item if $F_{j}\neq\emptyset$ then $ f_{ij}^{-1}[F_{j}]= F_{i}$.
\end{itemize}
\eroman
\end{definition}

Observe that the just defined notion of $r$-direct system differs from the definition of $l$-direct system above.  

Given a $r$-direct system of matrices $X$, a new matrix is defined as
\[
\PL(X) \coloneqq \langle \PL(\A_{i})_{i \in I}, \bigcup_{i \in I}F_{i}\rangle.
\]
\vspace{5pt}

\noindent
 Given a class $\mathsf{M}$ of matrices, $\PL^r(\mathsf{M})$ will denote the class of all P\l onka sums of $r$-directed systems of matrices in $\mathsf{M}$.

Given a logic $\vdash$ which is complete with respect to a class $\mathsf{M}$ of matrices, we set  $\mathsf{M^{\emptyset}}\coloneqq\mathsf{M}\cup\langle\A,\emptyset\rangle$, for any arbitrary $\A\in\Alg(\vdash)$.
The result which provides a complete matrix semantics for an arbitrary finitary right variable inclusion logic is the following
\begin{theorem}(\cite[Theorem 19]{BonzioPraBaldi}\label{completeness right}
Let $\vdash$ be a logic which is complete w.r.t. a class of non trivial matrices $\mathsf{M}$. Then $\vdash^{r}$ is complete w.r.t. $\PL^r(\mathsf{M^{\emptyset}})$. 
\end{theorem}

\begin{example}
Recall the situation of Example \ref{exa: pwk as left of cl}, and consider  the case in which $\vdash=\vdash_\CL$. Consider the class of matrices $\{\langle\B_2,1\rangle,\langle\mathbf{n},n\rangle\}$, where $\B_2$ is the two-element Boolean algebra, and $\mathbf{n}$ is the trivial algebra. Theorem \ref{completeness right} states that the following matrix is complete for $\vdash_{\mathsf{B}_3}$
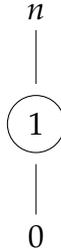
\begin{center}

\begin{tikzpicture}

  \node (n) at (-1.5,5) {$n$};
    \node (1) at (-1.5,3.5) {\circled{$1$}};
    \node (0) at (-1.5,2) {$0$};
         \draw (0) -- (1); 
 \draw (1) -- (n);

\end{tikzpicture}
\end{center}
\end{example}

\begin{definition}\label{def: logic with p-func.}
A logic $\vdash$ has a \emph{r-partition function} if there is a formula $x\ast y$, in which the variables $x$ and $y$ really occur, such that 
\benroman
\item$x,y\vdash x\ast y$, 
\item$x\ast y \vdash x$, 
\eroman
and the term operation $\ast$ is a partition function in every $\A\in\Alg(\vdash)$.

\end{definition}

\begin{remark}\label{rem: complete for rl lr}

Observe that, according with Theorem \ref{thm:completeness for left} and Theorem \ref{completeness right}, given $\mathsf{M^{\vdash}}$ a complete class of matrices for $\vdash$ containing $\langle\mathbf{n},n\rangle$ as only trivial matrix, it is always possible to obtain  a complete class of non trivial matrices $\mathsf{M}$ for $\vdash^{l}$, and  a complete class of  matrices $\mathsf{M}^{\star}$ for $\vdash^{r}$ containing $\langle\mathbf{n},n\rangle$ as only trivial matrix. Moreover, by applying again the mentioned theorems  to $\mathsf{M}$ and $\mathsf{M^{\star}}$ we have that $\PL^r(\mathsf{M}\cup\langle\mathbf{n},\emptyset\rangle)$ is complete for $\vdash^{lr}$ while $\PL^l(\mathsf{M^{\star}})$ is complete for $\vdash^{rl}$.
\end{remark}

In what follows, we write $\bullet$ to denote any (possibly empty) sequence of elements among $\{l,r\}$. So, $\vdash^{\bullet}$ will denote an arbitrary  logic obtained by replacing $\bullet$ with a sequence of elements among $\{l,r\}$.  We denote the length of a sequence $\bullet$ as $L(\bullet)$. 

The reading of a sequence $\bullet$ is from left to right. So, if $\bullet=u_{1}\dots u_{n}$ with ($u_{i}\in\{l,r\}$ for $1\leq i\leq n$) the logic $\vdash^{\bullet}$ is the logic obtained by applying the definition of $u_{m}$ to the logic $\vdash^{u_{1}\dots u_{m-1}}$ for every $1\leq m\leq n$.

An immediate  consequence of Remark \ref{rem: complete for rl lr} is that  $\vdash^{\bullet l}\geq\vdash^{\bullet lr}$ and $\vdash^{\bullet r}\geq\vdash^{\bullet rl}$. This fact will be useful for  the next sections.
From now on, unless stated otherwise, we assume that $\vdash$ is a finitary logic, and that it possesses a binary term $\pi(x,y)$ that behaves as a $r$-partition function for $\vdash^{r}$ and as  a $l$-partition function for $\vdash^{l}$. Observe that a great amount of logics share this feature. For instance, the term $\pi(x,y)=x\land(x\lor y)$,  is a partition function for classical and intutionistic logic, as well as for every substrucural and modal logic.

\section{Logics without antitheorems}\label{sec:sublogics without antitheorems}

In this section, given an antitheorem-free logic $\vdash$, we determine the number of the sublogics of variable inclusion of $\vdash$. Then, we investigate their position within the lattice of sublogics of $\vdash$.

\begin{lemma}\label{lemma generale su rl}
Let $\vdash$ be a logic without antitheorems. If \ $\Gamma\vdash^{rl\bullet^{\prime}}\varphi$, then there exists $\Delta\subseteq\Gamma$ such that $\Delta\vdash\varphi$ and $\Var(\Delta)=\Var(\varphi)$.
\end{lemma}
\begin{proof}
By induction on the length of $\bullet^{\prime}$.\\
(B). If $L(\bullet^{\prime})=0$ the proof is immediate, so it remains to consider $L(\bullet^{\prime})=1$. There are cases: (a) $\bullet^{\prime}=l$ or (b): $\bullet^{\prime}=r$. if (a) then $\Gamma\vdash^{\bullet rll}\varphi$ implies $\Gamma\vdash^{\bullet rl}\varphi$, so there exists $\Delta\subseteq\Gamma$ such that $\Delta\vdash^{\bullet r}\varphi$ and $\Var(\Delta)\subseteq\Var(\varphi)$. This implies $\Delta\vdash^{\bullet}\varphi$ and $\Var(\varphi)\subseteq\Var(\Delta)$. Now, suppose $\Delta\nvdash\varphi$. This implies $\Delta\nvdash^{l}\varphi$ and $\Delta\nvdash^{r}\varphi$, which is in contradiction with the fact that $\Delta\vdash^{\bullet}\varphi$. So $\Delta\vdash\varphi$.
The case of (b) is analogous.\\

(IND). Suppose the statement holds for $L(\bullet^{\prime})=n$ and consider $L(\bullet^{\prime})=n+1$. That is, $\bullet^{\prime}$ can be of the following forms: (a) $\bullet^{\prime}=s\cup\{l\}$ with $L(s)=n$, or (b) $\bullet^{\prime}=s\cup\{r\}$ with $L(s)=n$.
In the case of (a), as $\Gamma\vdash^{\bullet rl\bullet^{\prime}}\varphi$ we have that there exists $\Delta\subseteq\Gamma$ such that $\Delta\vdash^{\bullet rls}\varphi$ and $\Var(\Delta)\subseteq\Var(\varphi)$. As $L(s)=n$, by inductive hypothesis there exists $\Sigma\subseteq\Delta$ such that $\Sigma\vdash\varphi$ and $\Var(\Sigma)=\Var(\varphi)$. Observing that $\Sigma\subseteq\Delta\subseteq\Gamma$ we obtain our conclusion.
The case for (b) can be proved with the same strategy.

\end{proof}

The previous Lemma \ref{lemma generale su rl} has the following immediate consequences:

\begin{corollary}\label{cor: rl  bot senza antiteoremi}
Let $\vdash$ be a logic without antitheorems. Then
\benroman
\item   If \ $\Gamma\vdash^{\bullet rl\bullet^{\prime}}\varphi$ then there exists $\Delta\subseteq\Gamma$ such that $\Delta\vdash\varphi$ and $\Var(\Delta)=\Var(\varphi)$
\item $\vdash^{rl}{\leq}\vdash^{\bullet}$
\eroman
\end{corollary}

\begin{remark}\label{rem: l no antiteoremi}
Observe that every logic $\vdash^{\bullet}$ such that $l\in\bullet$ does not have antitheorems. Indeed, let $\vdash$ be a logic and suppose $\Sigma(x)$ is an antitheorem for $\vdash^{l}$. Let $X$ be a $l$-direct system of matrices such that 
\benroman
\item $I=\{i,j\}$ with $i\leq j$
\item $\langle\A_{i},F_{i}\rangle\in\Mod(\vdash)$ be non trivial
\item $\langle \A_{j},F_{j}\rangle$ such that $\A_{j}=\mathbf{n}, F_{j}=n$
\item $f_{ij}:\A_{i}\to\A_{j}$ be the unique homomorphism
\eroman 

Then by Theorem \ref{thm:completeness for left} $\PL(X)=\langle\A,F\rangle$ is a model of $\vdash^l$. The fact that $\Sigma(x)$ is an antitheorem for $\vdash^{l}$ implies $\Sigma(x)\vdash^{l} y$ for $y\in \Var$. Let now $h:\Fm\to\PL(\A_{i})_{i\in I}$ be such that $h(x)=n, h(y)=c$ with $c\in A_{i}\smallsetminus F_{i}$ (note that such $c$ exists as $A_{i}\neq F_{i}$). Then clearly $h(\Sigma(x))\subseteq F$, while $h(y)\notin F$, a contradiction.
\end{remark}

The following theorem characterizes the relation among the sublogics of variable inclusion of an antitheorem-free logic $\vdash$.

\begin{theorem}\label{the: lattice of log of var inc}
Let $\vdash\neq\vdash^{r},\vdash^{l}$ be a logic without antitheorems. The following relations hold:
\benroman
\item $\vdash^{l}{\nleq}\vdash^{r}$ and $\vdash^{r}{\nleq}\vdash^{l}$
\item $\vdash^{l}\cap\vdash^{r}{=}\vdash^{lr}\lneq\vdash^{l},\vdash^{r}$
\item $\vdash^{rl}=\vdash^{rl\bullet}{=}\vdash^{lrl\bullet}$
\eroman
\end{theorem}

\begin{proof}
(i) it immediately follows by noticing that $\pi(x,y)\vdash^{r}x$ while $\pi(x,y)\nvdash^{l}x$ and $x\vdash^{l}\pi(x,y)$ while $x\nvdash^{r}\pi(x,y)$.

(ii) 
As a direct consequence of Remark \ref{rem: complete for rl lr} we have $\vdash^{lr}\leq\vdash^{l}$.
 We now prove using contraposition that $\vdash^{lr}\leq\vdash^{r}$. So assume $\Gamma\nvdash^{r}\varphi$. There are cases, namely (1) $\Gamma\nvdash\varphi$ or (2) $\Var(\varphi)\nsubseteq\Var(\Gamma)$. (1) immediately implies $\Gamma\nvdash^{l}\varphi$, so $\Gamma\nvdash^{lr}\varphi$. If it is case of (2), assume towards a contradiction that $\Gamma\vdash^{lr}\varphi$. This entails that $\Gamma\vdash^{l}\varphi$ and that $\Var(\varphi)\subseteq\Var(\Gamma)$, which is a contradiction. So $\Gamma\nvdash^{lr}\varphi$. 
 
 Now, $\vdash^{l}\cap\vdash^{r}\leq\vdash^{lr}$ follows by noticing that in the lattice of sublogics of $\vdash$ it holds $\vdash^{l}\land\vdash^{r}=\vdash^{l}\cap\vdash^{r}$, and so, as $\vdash^{lr}\leq\vdash^{r},\vdash^{l}$ it follows $\vdash^{l}\cap\vdash^{r}\leq\vdash^{lr}$. For the other direction, assume $\Gamma\vdash^{lr}\varphi$. This entails $\Gamma\vdash^{l}\varphi$ with $\Var(\varphi)\subseteq\Var(\Gamma)$. Furthermore, as $\vdash^{l}\leq\vdash$, we have $\Gamma\vdash\varphi$ which finally entails $\Gamma\vdash^{r}\varphi$.
 
 Moreover, the fact that $\pi(x,y)\vdash^{r}x$ while $\pi(x,y)\nvdash^{lr}x$ and $x\vdash^{l}\pi(x,y)$ while $x\nvdash^{lr}\pi(x,y)$ proves the desired proper inequality.

(iii) That $\vdash^{rl\bullet}\leq\vdash^{rl}$ follows again by remark \ref{rem: complete for rl lr}. That $\vdash^{rl}\leq\vdash^{rl\bullet}$ follows immediately from  Corollary \ref{cor: rl  bot senza antiteoremi}. Now we prove $\vdash^{lrl\bullet}\leq\vdash^{rl}$. To this end, assume  $\Gamma\vdash^{lrl\bullet}\varphi$. By Lemma \ref{lemma generale su rl} we have that there exists $\Delta\subseteq\Gamma$ such that $\Delta\vdash\varphi$ and $\Var(\Delta)=\Var(\varphi)$. So, it follows $\Delta\vdash^{rl}\varphi$ and, by monotonicity, we obtain $\Gamma\vdash^{rl}\varphi$. The fact that $\vdash^{rl}\leq\vdash^{lrl\bullet}$ is a consequence of Corollary \ref{cor: rl  bot senza antiteoremi}.
\end{proof}

\begin{remark}
Observe that if a logic $\vdash$ has a theorem $\varphi$, then $\vdash^{rl}\lneq\vdash^{lr}$. Indeed it is immediate to verify that $\pi(x,y)\vdash^{lr}\varphi(x)$ while $\pi(x,y)\nvdash^{rl}\varphi(x)$.
\end{remark}\noindent

The following corollary summarizes the results of this section.

\begin{corollary}

Let $\vdash$ be a logic with a partition function and without antitheorems. Then the following holds:
\benroman
\item  there are at most four proper sublogics of variable inclusion of $\vdash$.
\item The sublattice of $\mathcal{L}^\vdash$ generated by $\mathcal{SV}(\vdash)$ has (at most) six elements, and it is represented by the following Figure \ref{lattice log. pic}.
\eroman
\end{corollary}

\begin{center}
 \begin{tikzpicture}[scale=1.6]\label{lattice log. pic}
 
      \node (ciu) at (1,5) {$\vdash$};
   \node (one) at (1,4) {$\vdash^l\lor\vdash^r$};
          \node (>s1) at (0,3) {$\vdash^{l}$};
           \node (>d1) at (2,3) {$\vdash^{r}$};
    \node (<<1) at (1,2) {$\vdash^{lr}{=}\vdash^{l}\cap\vdash^{r}$};
    \node (<1) at (1,1) {$\vdash^{rl}{=}\vdash^{rl\bullet}{=}\vdash^{lrl\bullet}$};
       \draw[thick] (one) -- (ciu); 
        \draw[thick] (one) -- (>s1); 
 \draw[thick] (one) -- (>d1); 
 \draw[thick] (>d1) -- (<<1); 
 \draw [thick](>s1) -- (<<1); 
 \draw [thick](<1) -- (<<1); 
\node[align=center,font=\bfseries, yshift=2em] (title) 
   at (current bounding box.north)
    { Figure \ref{lattice log. pic}};
    \end{tikzpicture}
\end{center}

\section{Logics with antitheorems}

We now turn to the case in which the logic $\vdash$ does posses an antitheorem $\Sigma(x)$. In the next Theorem \ref{th lat logvarinc with incons} we assume w.l.o.g. $\Sigma(x)=\{\epsilon_{1}(x),\dots,\epsilon_{n}(x)\}$.

\begin{theorem} \label{th lat logvarinc with incons}
Let $\vdash$ be a logic with antitheorems. Then the following relations hold
\benroman
\item $\vdash^{rl}\nleq\vdash^{lr}$ and $\vdash^{lr}\nleq\vdash^{rl}$

\item $\vdash^{l}\cap\vdash^{r}{\gneq}\vdash^{lr},\vdash^{rl}$
\item $\vdash^{rlr}{\lneq}\vdash^{rl}$ and $\vdash^{lrl}{\lneq}\vdash^{lr}$
\item $\vdash^{rlr}{\lneq}\vdash^{lr}\cap\vdash^{rl}$
\item $\vdash^{lrl}{=}\vdash^{lrlr}{=}\vdash^{rlrl}{\lneq}\vdash^{rlr}$.
\item $\vdash^{rlrl\bullet}{=}\vdash^{lrl\bullet}$

\eroman
where $\bullet$ denotes any (possibly empty) sequence of elements among $\{l,r\}$.

\end{theorem}

\begin{proof}

$(i)$. 
Firstly we show $\vdash^{rl}\nleq\vdash^{lr}$. To this end it is immediate to verify that $\Sigma(x)\vdash^{rl}\pi(x,y)$ while $\Sigma(x)\nvdash^{lr}\pi(x,y)$. 

For the other inequality, 
first observe that \[\Var(\pi(y,z))\subseteq\Var(y,\pi(\epsilon_{1}(x),z),\dots,\pi(\epsilon_{n}(x),z))\] and, moreover 
\[y,\pi(\epsilon_{1}(x),z),\dots,\pi(\epsilon_{n}(x),z)\vdash^{l
}\pi(y,z),\]
as $y\vdash^{l}\pi(y,z)$ and $\{y\}\subseteq\{y,\pi(\epsilon_{1}(x),z),\dots,\pi(\epsilon_{n}(x),z)\}$. So, this proves
\[y,\pi(\epsilon_{1}(x),z),\dots,\pi(\epsilon_{n}(x),z)\vdash^{lr}\pi(y,z).\]

This, together with the fact that for no $\Delta\subseteq\{y,\pi(\epsilon_{1}(x),z),\dots,\pi(\epsilon_{n}(x),z)\}$ it holds $\Delta\vdash^{r}\pi(y,z)$ and $\Var(\Delta)\subseteq\{y,z\}$ shows
\[y,\pi(\epsilon_{1}(x),z),\dots,\pi(\epsilon_{n}(x),z)\nvdash^{rl}\pi(y,z),\]
as desired.

$(ii)$. We first prove $\vdash^{l}\cap\vdash^{r}\geq\vdash^{lr},\vdash^{rl}$. Let $\Gamma\vdash^{lr}\varphi$, then, as $\vdash^{l}$ does not have antitheorems, it must be that $\Gamma\vdash^{l}\varphi$ and $\Var(\varphi)\subseteq\Var(\Gamma)$. This, together with $\vdash^{l}\leq\vdash$ entails $\Gamma\vdash\varphi$, so $\Gamma\vdash^{r}\varphi$. So, $\Gamma\vdash^{l}\cap\vdash^{r}\varphi$. That $\vdash^{l}\cap\vdash^{r}\geq\vdash^{rl}$ is proved in the same way. 

As the inferences described in point  $(i)$ hold both in $\vdash^{l}$ and $\vdash^{r}$, we obtain $\vdash^{lr},\vdash^{rl}\lneq\vdash^{l}\cap\vdash^{r}$.

$(iii)$. The fact that $\vdash^{rlr}\leq\vdash^{rl}$ and $\vdash^{lrl}\leq\vdash^{lr}$is a direct consequence of Remark \ref{rem: complete for rl lr}.

This, together with the fact that
\[y,\pi(\epsilon_{1}(x),z),\dots,\pi(\epsilon_{n}(x),z)\nvdash^{lrl}\pi(y,z)\]
and $\Sigma(x)\nvdash^{rlr}\pi(x,y)$ proves the desired proper inequalities.

$(iv)$. We first prove $\vdash^{rlr}\leq\vdash^{lr}\cap\vdash^{rl}$. That $\vdash^{rlr}\leq\vdash^{rl}$ follows, again by Remark \ref{rem: complete for rl lr}. Consider $\Gamma\vdash^{rlr}\varphi$, so, as $\vdash^{rl}$ does not have antitheorems, $\Gamma\vdash^{rl}\varphi$ with $\Var(\varphi)\subseteq\Var(\Gamma)$. This entail that there exists $\Delta\subseteq\Gamma$, $\Delta\vdash^{r}\varphi$ and $\Var(\Delta)\subseteq\Var(\varphi)$. As, $\vdash^{r}\leq\vdash$ we obtain $\Delta\vdash\varphi$, so $\Delta\vdash^{l}\varphi$ which, by monotonicity entails $\Gamma\vdash^{l}\varphi$. Recalling that $\Var(\varphi)\subseteq\Var(\Gamma)$ we conclude $\Gamma\vdash^{lr}\varphi$.

The proper inclusion is proved by noticing that $\Sigma(x)\vdash^{lr}\pi(x,y)$, $\Sigma(x)\vdash^{rl}\pi(x,y)$ while $\Sigma(x)\nvdash^{rlr}\pi(x,y)$.

$(v)$. As by remark \ref{rem: l no antiteoremi} $\vdash^{l},\vdash^{lr},\vdash^{rl}$ are logics without antitheorems, then by Lemma \ref{lemma generale su rl} we know that $\Gamma\vdash^{lrl}\varphi$ entails that there exists $\Delta\subseteq\Gamma$, $\Delta\vdash\varphi$ and $\Var(\varphi)=\Var(\Delta)$ (the same holds for $\vdash^{lrlr}$ and $\vdash^{rlrl}$). As this immediately implies $\Delta\vdash^{lrlr}\varphi$ and $\Delta\vdash^{rlrl}\varphi$, by monotonicity we conclude $\Gamma\vdash^{lrlr}\varphi$ and $\Gamma\vdash^{rlrl}\varphi$, so $\vdash^{lrlr}=\vdash^{rlrl}=\vdash^{lrl}$.

 It only remains to prove that $\vdash^{rlrl}\lneq\vdash^{rlr}$. To this end, it suffices to note that $\pi(y,z),\Sigma(x)\vdash^{rlr}\pi(y,x)$ while  $\pi(y,z),\Sigma(x)\nvdash^{rlrl}\pi(y,x)$.

 $(vi)$. The equality $\vdash^{rlrl\bullet}=\vdash^{lrl\bullet}$ is a straightforward application of Lemma \ref{lemma generale su rl}, using the same strategy of point (v).
\end{proof}

The following corollary summarizes the results of the section:

\begin{corollary}

Let $\vdash$ be a logic with a partition function and antitheorems, then 

\benroman

\item  there are at most 6 proper sublogics of variable inclusion of $\vdash$.

\item the sublattice of $\mathcal{L}^\vdash$ generated by $\mathcal{SV}(\vdash)$ has (at most) 11 elements, and it is represented by the following  Figure \ref{lattice log. incterm pic}

\eroman 

\end{corollary}

\begin{center}
 \begin{tikzpicture}[scale=1.3]\label{lattice log. incterm pic}
   
   \node (one) at (1,6) {$\vdash$};
      \node (ljoinr) at (1,5) {$\vdash^l\lor\vdash^r$};
          \node (l) at (0,4) {$\vdash^{l}$};
           \node (r) at (2,4) {$\vdash^{r}$};
                      \node (join rl lr) at (1,2) {$\vdash^{rl}\lor\vdash^{lr}$};
    \node (cap1) at (1,3) {$\vdash^{r}\cap \vdash^{l}$};
     \node (lr) at (0,1) {$\vdash^{lr}$};
           \node (rl) at (2,1) {$\vdash^{rl}$};
               \node (cap2) at (1,0) {$\vdash^{rl}\cap\vdash^{lr}$};
               \node (rlr) at (1,-1) {$\vdash^{rlr}$};
               \node (rlrl) at (1,-2) {$\vdash^{rlrl\bullet}{=}\vdash^{lrl}{=}\vdash^{lrl\bullet}$};

        \draw [thick](ljoinr) -- (one);

        \draw [thick](ljoinr) -- (l); 
 \draw [thick](ljoinr) -- (r); 
 \draw [thick](r) -- (cap1); 
 \draw [thick](l) -- (cap1); 
  \draw [thick](join rl lr) -- (cap1); 
 \draw [thick](join rl lr) -- (lr); 
 \draw [thick](join rl lr) -- (rl); 
  \draw [thick](rl) -- (cap2); 
  \draw [thick](lr) -- (cap2); 
    \draw[thick] (cap2) -- (rlr); 
  \draw[thick] (rlrl) -- (rlr); 

\node[align=center,font=\bfseries, yshift=2em] (title) 
   at (current bounding box.north)
    { Figure \ref{lattice log. incterm pic}};
    \end{tikzpicture}
\end{center}

\section{The sublogics of variable inclusion of classical logic}
In this final section we briefly investigate the lattice of logics of variable inclusion of Classical logic. 
 As we already noticed, only three proper sublogics of variable inclusion of classical logics have already been investigated, namely $\PWK,\mathsf{B}_3$ and $\vdash_{\mathsf{K^{w}_{4n}}}$.

Let now consider $\vdash_\CL$, the logic defined by the matrix $\langle\B_{2}, 1\rangle$ where $\B_{2}$ is the two-elements Boolean algebra. The matrices that defines the sublogics of variable inclusion of classical logic are as follows:

 
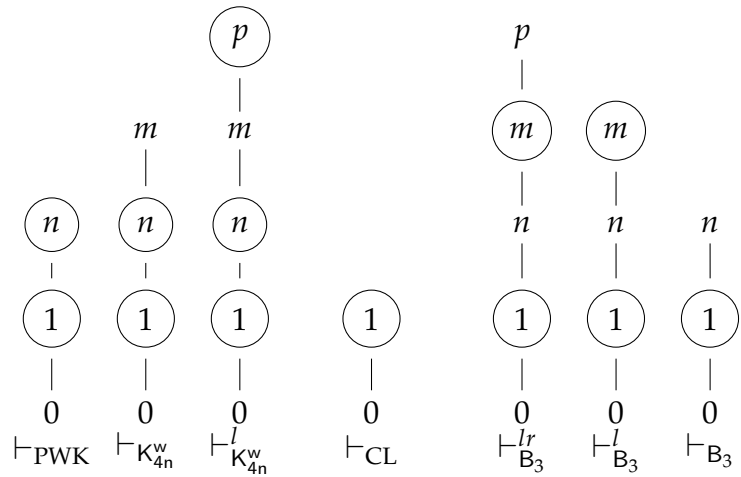
\begin{center}

\begin{tikzpicture}[scale=2.5]\label{example: lattice log.pic}

    \node (1) at (0.2,3.5) {\circled{$1$}};
    \node (0) at (0.2,3) {$0$};
        \node (tit) at (0.2,2.8) {$\vdash_{\CL}$};
         \draw (0) -- (1);

   \node (n) at (-1.5,4) {\circled{$n$}};
    \node (1) at (-1.5,3.5) {\circled{$1$}};
    \node (0) at (-1.5,3) {$0$};
        \node (tit) at (-1.5,2.8) {$\vdash_{\PWK}$};
         \draw (0) -- (1); 
 \draw (1) -- (n); 
        
       \node (m) at (-1,4.5) {$m$};
   \node (n) at (-1,4) {\circled{$n$}};
    \node (1) at (-1,3.5) {\circled{$1$}};
    \node (0) at (-1,3) {$0$};
        \node (tit) at (-1,2.8) {$\vdash_{\mathsf{K^{w}_{4n}}}$};

 \draw (0) -- (1); 
 \draw (1) -- (n); 
  \draw (n) -- (m); 
 
      \node (p) at (-0.5,5) {\circled{$p$}};
    \node (m) at (-0.5,4.5) {$m$};
    \node (n) at (-0.5,4) {\circled{$n$}};
    \node (1) at (-0.5,3.5) {\circled{$1$}};
    \node (0) at (-0.5,3) {$0$};
        \node (tit) at (-0.5,2.8) {$\vdash_{\mathsf{K^{w}_{4n}}}^l$};

 \draw (0) -- (1); 
 \draw (1) -- (n); 
  \draw (n) -- (m); 
      \draw (m) -- (p); 

%
%

         
   \node (n) at (2,4) {$n$};
    \node (1) at (2,3.5) {\circled{$1$}};
    \node (0) at (2,3) {$0$};
        \node (tit) at (2,2.8) {$\vdash_{\mathsf{B}_3}$};

 \draw (0) -- (1); 
 \draw (1) -- (n); 
   
           \node (p) at (1.5,4.5) {\circled{$m$}};
   \node (n) at (1.5,4) {$n$};
    \node (1) at (1.5,3.5) {\circled{$1$}};
    \node (0) at (1.5,3) {$0$};
        \node (tit) at (1.5,2.8) {$\vdash_{\mathsf{B}_3}^l$};

 \draw (0) -- (1); 
 \draw (1) -- (n); 
  \draw (p) -- (n);

     \node (p) at (1,5) {$p$};
    \node (m) at (1,4.5) {\circled{$m$}};
    \node (n) at (1,4) {$n$};
    \node (1) at (1,3.5) {\circled{$1$}};
    \node (0) at (1,3) {$0$};
        \node (tit) at (1,2.8) {$\vdash_{\mathsf{B}_3}^{lr}$};

 \draw (0) -- (1); 
 \draw (1) -- (n); 
  \draw (n) -- (m); 
    \draw (m) -- (p);



    \end{tikzpicture}
    
\end{center}

Letting $\pi(x,y)=x\land(x\lor y)$ it is not difficult to apply Theorem \ref{th lat logvarinc with incons}, and to observe that there are 6 proper sublogics of variabale inclusion of $\CL$. Moreover, the lattice generated by $\mathcal{SV}(\vdash_\CL)$ is the following:

\begin{center}
 \begin{tikzpicture}[scale=1.3]\label{lattice log. cl pic}
   
   \node (vdash) at (1,5) {$\vdash_\CL$};
          \node (l) at (0,4) {$\vdash_\PWK$};
           \node (r) at (2,4) {$\vdash_{\mathsf{B}_3}$};
                      \node (join rl lr) at (1,2) {$\vdash_{\mathsf{B}_3}^l\lor\vdash_{\mathsf{K^{w}_{4n}}}$};
    \node (cap1) at (1,3) {$\vdash_{\mathsf{B}_3}\cap \vdash_\PWK$};
     \node (lr) at (0,1) {$\vdash_{\mathsf{K^{w}_{4n}}}$};
           \node (rl) at (2,1) {$\vdash_{\mathsf{B}_3}^l$};
               \node (cap2) at (1,0) {$\vdash_{\mathsf{B}_3}^l\cap\vdash_{\mathsf{K^{w}_{4n}}}$};               \node (rlr) at (1,-1) {$\vdash_{\mathsf{B}_3}^{lr}$};
               \node (rlrl) at (1,-2) {$\vdash_{\mathsf{K^{w}_{4n}}}^l$};

        \draw [thick](vdash) -- (l); 
 \draw [thick](vdash) -- (r); 
 \draw [thick](r) -- (cap1); 
 \draw [thick](l) -- (cap1); 
  \draw [thick](join rl lr) -- (cap1); 
 \draw [thick](join rl lr) -- (lr); 
 \draw [thick](join rl lr) -- (rl); 
  \draw [thick](rl) -- (cap2); 
  \draw [thick](lr) -- (cap2); 
    \draw[thick] (cap2) -- (rlr); 
  \draw[thick] (rlrl) -- (rlr); 

\node[align=center,font=\bfseries, yshift=2em] (title) 
   at (current bounding box.north)
    { Figure \ref{lattice log. cl pic}};
    \end{tikzpicture}
\end{center}


\begin{thebibliography}{10}

\bibitem{beall2016off}
Jc~Beall.
\newblock Off-topic: A new interpretation of weak-kleene logic.
\newblock {\em The Australasian Journal of Logic}, 13(6), 2016.

\bibitem{bergmann}
M.~Bergmann.
\newblock {\em An {I}ntroduction to {M}any-{V}alued and {F}uzzy {L}ogic}.
\newblock Cambridge {U}niversity {P}ress, {C}ambridge, 2008.

\bibitem{blok2006equivalence}
Willem~J Blok and Bjarni J{\'o}nsson.
\newblock Equivalence of consequence operations.
\newblock {\em Studia Logica}, 83(1-3):91--110, 2006.

\bibitem{BP89}
W.J. Blok and D.~Pigozzi.
\newblock {\em Algebraizable logics}, volume 396 of {\em Mem. Amer. Math. Soc.}
\newblock A.M.S., 1989.

\bibitem{Bochvar}
D.A. Bochvar.
\newblock On a three-valued calculus and its application in the analysis of the
  paradoxes of the extended functional calculus.
\newblock {\em Mathematicheskii Sbornik}, 4:287--308, 1938.

\bibitem{Bonzio16}
S.~Bonzio, J.~Gil-F\'erez, F.~Paoli, and L.~Peruzzi.
\newblock On {P}araconsistent {W}eak {K}leene {L}ogic: axiomatization and
  algebraic analysis.
\newblock {\em Studia Logica}, 105(2):253--297, 2017.

\bibitem{BonzioMorascoPrabaldi}
S.~Bonzio, T.~Moraschini, and M.~Pra~Baldi.
\newblock Logics of left variable inclusion and {P}\l onka sums of matrices.
\newblock Submitted manuscript, 2018.

\bibitem{BonzioPraBaldi}
S.~Bonzio and M.~Pra~Baldi.
\newblock Containment logics and {P}\l onka sums of matrices.
\newblock Submitted manuscript, 2018.

\bibitem{BuSa00}
S.~Burris and H.~P. Sankappanavar.
\newblock {\em A course in {U}niversal {A}lgebra}.
\newblock Available in internet
  \url{https://www.math.uwaterloo.ca/~snburris/htdocs/ualg.html}, the
  millennium edition, 2012.

\bibitem{CampercholiRaftery}
Miguel~A. Campercholi and James~G. Raftery.
\newblock Relative congruence formulas and decompositions in quasivarieties.
\newblock {\em Algebra universalis}, 78(3):407--425, 2017.

\bibitem{CiuniCarrara}
R.~Ciuni and M.~Carrara.
\newblock Characterizing logical consequence in paraconsistent weak kleene.
\newblock In L.~Felline, A.~Ledda, F.~Paoli, and E.~Rossanese, editors, {\em
  New Developments in Logic and the Philosophy of Science}. College, London,
  2016.

\bibitem{cc}
M.~Coniglio and M.I. Corbal\'an.
\newblock Sequent calculi for the classical fragment of bochvar and halld\'en's
  nonsense logics.
\newblock {\em Proceedings of LSFA 2012}, pages 125--136, 2012.

\bibitem{Cz01}
J.~Czelakowski.
\newblock {\em Protoalgebraic logics}, volume~10 of {\em Trends in
  Logic---Studia Logica Library}.
\newblock Kluwer Academic Publishers, Dordrecht, 2001.

\bibitem{Font16}
J.M. Font.
\newblock {\em Abstract Algebraic Logic: An Introductory Textbook}.
\newblock College Publications, 2016.

\bibitem{font2015m}
Josep~Maria Font and Tommaso Moraschini.
\newblock M-sets and the representation problem.
\newblock {\em Studia Logica}, 103(1):21--51, 2015.

\bibitem{Hallden}
S.~Halld\'en.
\newblock {\em The Logic of Nonsense}.
\newblock Lundequista Bokhandeln, Uppsala, 1949.

\bibitem{Kleene}
S.C. Kleene.
\newblock {\em Introduction to Metamathematics}.
\newblock North Holland, Amsterdam, 1952.

\bibitem{LaPr17}
T.~L\'{a}vi\v{c}ka and A.~P\v{r}enosil.
\newblock Protonegational logics and inconsistency lemmas.
\newblock In {\em Proocedings of ManyVal 2017}, November 2017.

\bibitem{Parry}
William~Tuthill Parry.
\newblock {\em Implication}.
\newblock PhD Thesis, Harvard University, 1932.

\bibitem{Petrukhin2}
Y.I. Petrukhin.
\newblock Natural deduction for four-valued both regular and monotonic logics.
\newblock {\em Logic and Logical Philosophy}, 27:53--66, 2018.

\bibitem{Plo67}
J.~P\l{}onka.
\newblock On a method of construction of abstract algebras.
\newblock {\em Fundamenta Mathematicae}, 61(2):183--189, 1967.

\bibitem{Plo67a}
J.~P\l{}onka.
\newblock On distributive quasilattices.
\newblock {\em Fundamenta Mathematicae}, 60:191--200, 1967.

\bibitem{Romanowska92}
J.~P{\l}onka and A.~Romanowska.
\newblock Semilattice sums.
\newblock {\em Universal Algebra and Quasigroup Theory}, pages 123--158, 1992.

\bibitem{JGR13}
J.~G. Raftery.
\newblock Inconsistency lemmas in algebraic logic.
\newblock {\em Mathematical Logic Quaterly}, 59(6):393--406, 2013.

\bibitem{romanowska2002modes}
A.B. Romanowska and J.D.H. Smith.
\newblock {\em Modes}.
\newblock World Scientific, 2002.

\bibitem{Szmuc}
D.~Szmuc.
\newblock Defining {LFI}s and {LFU}s in extensions of infectious logics.
\newblock {\em Journal of {A}pplied non {C}lassical {L}ogics}, 26(4):286--314,
  2016.

\bibitem{Szmuctruth}
Damian Szmuc, Bruno~Da Re, and Federico Pailos.
\newblock Theories of truth based on four-valued infectious logics.
\newblock {\em Logic Journal of the IGPL}, forthcoming.

\bibitem{Tomova}
N.E. Tomova.
\newblock About four-valued regular logics.
\newblock {\em Logical {I}nvestigations}, 15:223--228, 2009.
\newblock (in Russian).

\bibitem{W88}
R.~W{\'o}jcicki.
\newblock {\em Theory of logical calculi. {B}asic theory of consequence
  operations}, volume 199 of {\em Synthese Library}.
\newblock Reidel, Dordrecht, 1988.

\end{thebibliography}

\end{document}